\documentclass[12pt,twoside]{article}
\usepackage{amsmath, amsthm, graphicx}
\usepackage{amsfonts}
\usepackage[english]{babel}
\usepackage{float}

\newtheorem{theorem}{Theorem}[section]

\newtheorem{lemma}[theorem]{Lemma}

\date{}

\begin{document}

\title{{\Large\bf Maps preserving the $\varepsilon$-pseudo spectrum of some product of operators}}

\author{{\normalsize\sc H. Bagherinejad$^{1}$, A. Iloon Kashkooly$^{1}$, R. Parvinianzadeh$^{1,}$\footnote{ Corresponding author: }}\\[0.5cm]
 $^{1}$Department of Mathematics, University of Yasouj,
 Yasouj, Iran\\
 \footnote{E-mail addresses: bagheri1361h@gmail.com (H. Bagherinejad), kashkooly@yu.ac.ir (A. Iloon Kashkooly), r.parvinian@yu.ac.ir (R. Parvinianzadeh).}
 }
\maketitle

{\footnotesize  {\bf Abstract} Let $B(H)$ be the algebra of all bounded linear operators on infinite-dimensional complex Hilbert space $H$. For $T, S \in B(H)$ denote
by $T\bullet S=TS+ST^{\ast}$ and $[T\circ S]_{\ast}=TS-ST^{\ast}$ the Jordan $\ast$-product and the skew Lie product of $T$ and $S$, respectively. Fix $\varepsilon > 0$ and $T \in B(H)$, let $\sigma_{\varepsilon}(T)$ denote the $\varepsilon$-pseudo spectrum of $T$. In this paper, we describe bijective maps $\varphi$ on $B(H)$ which satisfy
\begin{align*}
\sigma_{\varepsilon}([T_{1}\bullet T_{2},T_{3}]_{\ast})=\sigma_{\varepsilon}([\varphi(T_{1})\bullet \varphi(T_{2}),\varphi(T_{3})]_{\ast}),
\end{align*}
for all $T_{1}, T_{2}, T_{3} \in B(H)$. We also characterize bijective maps $\varphi: B(H) \rightarrow B(H)$ that satisfy
\begin{align*}
\sigma_{\varepsilon}(T_{1}\diamond T_{2}\circ_{\ast} T_{3})=\sigma_{\varepsilon}(\varphi(T_{1})\diamond \varphi(T_{2})\circ_{\ast} \varphi(T_{3})),
\end{align*}
for all $T_{1}, T_{2}, T_{3} \in B(H)$, where $T_{1}\diamond T_{2}=T_{1}T_{2}^{\ast}+T_{2}^{\ast}T_{1}$ and $T_{1}\circ_{\ast} T_{2}=T_{1}T_{2}^{\ast}-T_{2}T_{1}$.\\


}

\section{\normalsize\bf Introduction and Background}

Preserver problems, in the most general setting, demands
the characterization of maps between algebras that leave a
certain property, a particular relation, or even a subset
invariant. This subject is very old and goes back well over a century to the so-called first linear preserver problem, due to Frobenius \cite{fro}, who characterized linear maps that preserve the determinant of matrices. The aforesaid Frobenius’ work was generalized by J. Dieudonné \cite{die}, who characterized linear maps preserving singular matrices. The goal is to describe the general form of maps between two Banach algebras that preserve a certain property, or a certain class of elements,
or a certain relation. One of the most famous problems in this direction is Kaplansky’s problem
\cite{kap} asking whether every surjective unital invertibility preserving linear map between two semi-simple
Banach algebras is a Jordan homomorphism.
His question was motivated by two classical results, the result of Marcus and Moyls \cite{mar} on linear maps preserving eigenvalues of matrices and the Gleason-Kahane-Zelazko theorem \cite{gli,kah} stating that that Every unital invertibility preserving linear functional on a unital complex Banach algebra is necessarily multiplicative. This result was obtained independently by Gleason in \cite{gli} and Kahane-Zelazko in \cite{kah}, and was refined by Zelazko in \cite{zel}. In the non-commutative case,
the best known results so far are due to Aupetit \cite{Aup} and Sourour \cite{Sourour}. For other preserver problems on different types of products on matrices and operators, one may see \cite{Bourhim2,jaf,jaf2,li,parvin} and their references.\\

Throughout this paper, let $\varepsilon$ be a fixed positive real number, and $B(H)$ stands for the algebra of all bounded linear operators
acting on an infinite dimensional complex Hilbert space $H$ and its unit will be denoted by $I$. Let $ B_{s}(H)$, $B_{a}(H)$ and $P(H)$ be the set of all self-adjoint, the set of all anti-self-adjoint and the set of
all projections operators in $B(H)$, respectively. The inner
product of $H$ will be denoted by $\left \langle , \right \rangle$ and we write $Z(B(H))$ for the center of $B(H)$. For an operator $T \in B(H)$, the spectrum, the adjoint and the transpose of $T$ relative to an arbitrary but fixed orthogonal basis of $H$ are denoted by $\sigma(T)$, $T^{\ast}$ and $T^{t}$, respectively. Let $Tr T$ denote the trace of a finite rank operator $T$.
For a fixed positive real number $\varepsilon >0$, the $\varepsilon$-pseudo spectrum of $T$, $\sigma_{\varepsilon}(T)$, is defined by
$$\sigma_{\varepsilon}(T)=\cup \{\sigma (T+A): A\in B(H), \|A\|\leq\varepsilon\}$$
and coincides with the set
$$\{\lambda \in \mathbb{C}: \|(\lambda I- T)^{-1}\|\geq \varepsilon^{-1}\}$$
with the convention that $\|(\lambda I- T)^{-1}\|=\infty$ if $\lambda \in \sigma(T)$. It follows from the
upper-semi continuity of the spectrum that the intersection of all the pseudo spectra is the spectrum,
$$\sigma(A)=\bigcap_{\varepsilon > 0} \sigma_{\varepsilon}(A).$$
For more information about these notions, one can see \cite{tre}.

The study of linear and nonlinear Pseudo spectra preserver problems attracted the attention of a number of authors. Mainly, several authors described maps on matrices or
operators that preserve the $\varepsilon$-pseudo spectral radius and the $\varepsilon$-pseudo spectrum of different kinds of products; see for instance \cite{abdel,Ben1,Ben2,Ben3,cu2,cu3,kum} and the references therein. The aim of this paper is to characterize mappings on $B(H)$ that preserve the $\varepsilon$-pseudo spectrum of different kinds of mixed product of operators.\\

In the following lemma, we collect some known properties of the $\varepsilon$-pseudo spectrum which are needed in the proof of the
main results. For any $z \in \mathbb{C}$ and $r> 0$, let $D(z, r)$ be the disc of $\mathbb{C}$
centered at $z$ and of radius $r$.

\begin{lemma} (See \cite{cu2,tre}.) For an operator $T\in B(H)$ and $\alpha > 0$, the following statements hold.\\
(1) $\sigma(T) + D(0, \varepsilon) \subseteq \sigma_{\varepsilon}(T)$.\\
(2) If $T$ is normal, then $\sigma_{\varepsilon}(T)=\sigma(T) + D(0, \varepsilon)$.\\
(3) For any $\alpha \in \mathbb{C}, \sigma_{\varepsilon}(T+\alpha I)=\alpha + \sigma_{\varepsilon}(T)$.\\
(4) For any nonzero $\alpha \in \mathbb{C}, \sigma_{\varepsilon}(\alpha T)=\alpha \sigma_{\frac{\varepsilon}{|\alpha|}}(T)$.\\
(5) For any $\alpha \in \mathbb{C}$, we have $\sigma_{\varepsilon}(T)= D(\alpha, \varepsilon)$ if and only if $T=\alpha I$.\\
(6) $\sigma_{\varepsilon}(T^{t})=\sigma_{\varepsilon}(T)$, where $T^{t}$ denotes the transpose of $T$ relative to an arbitrary but fixed orthonormal basis
of $H$.\\
(7) For every unitary operator $U\in B(H)$, we have $\sigma_{\varepsilon}(UTU^{\ast})=\sigma_{\varepsilon}(T)$.\\
(8) For every conjugate unitary operator $U$, we have $\sigma_{\varepsilon}(UTU^{\ast})=\sigma_{\varepsilon}(T^{\ast})$.\\
\end{lemma}

The following lemma describes the spectrum of the skew Lie product $[x\otimes x,T]_{\ast}$ for any
nonzero vector $x \in H$ and operator $T \in B(H)$.

\begin{lemma} \label{l2} (See \cite[Corollary 2.1]{al}.)
Let $T \in B(H)$ and $x \in H$ be a nonzero vector. Then
$$\sigma(T(x\otimes x)+(x\otimes x)T)=\{0, \left \langle Tx,x \right \rangle \pm \sqrt{ \left \langle T^{2}x,x\right \rangle } ~ \}.$$
\end{lemma}

The next lemma gives necessary and sufficient conditions for two operators to
be the same.

\begin{lemma} \label{l2} (See \cite[Lemma 2.2]{al}.)
Let $T$ and $S$ be in $B(H)$. Then the following statements are
equivalent.\\
$(1)$ $T=S$.\\
$(2)$ $\sigma([A,T]_{\ast})=\sigma([A,S]_{\ast})$ for every operator $A\in B(H)$.\\
$(2)$ $\sigma([A,T]_{\ast})=\sigma([A,S]_{\ast})$ for every operator $A\in B_{a}(H)$.
\end{lemma}

The following theorem will be useful in the proofs of the main results.

\begin{theorem}\label{t1} (See \cite[Theorem 3.3]{cu2}.)
Let $\varepsilon >0$. Then a surjective map $\varphi: B_{s}(H) \rightarrow B_{s}(H)$ satisfies
$$\sigma_{\varepsilon}(TS+ST)=\sigma_{\varepsilon}(\varphi(T)\varphi(S)+\varphi(S)\varphi(T))$$
if and only if there exists a unitary operator $U\in B(H)$ such that $\varphi$ has the form
 $T \rightarrow \mu UTU^{\ast}$ or $T \rightarrow \mu UT^{t}U^{\ast}$,
where $\mu \in \{-1,1\}$.
\end{theorem}

\section{ \normalsize\bf  Main Results}

The following theorem is one of the purposes of this paper.

\begin{theorem}\label{l2}
Suppose that a bijective map $\varphi: B(H) \rightarrow B(H)$ satisfies
\begin{align*}
\sigma_{\varepsilon}([T_{1}\bullet T_{2},T_{3}]_{\ast})=\sigma_{\varepsilon}([\varphi(T_{1})\bullet \varphi(T_{2}),\varphi(T_{3})]_{\ast}),~~ (T_{1}, T_{2}, T_{3} \in B(H)).
\end{align*}
Then there exist an invertible operator $S\in B(H)$ and a unitary operator $U\in B(H)$ such that
 $\varphi(T)=SUTU^{\ast}$ or $\varphi(T)=SUT^{t}U^{\ast}$ for every $T \in B(H)$.
\end{theorem}

\begin{proof}
The proof of it will be completed after checking several claims.\\

{\bf Claim 1.}  $\varphi(iI)^{\ast}=-\varphi(iI)\in Z(B(H))$.\\

Since $\varphi$ is surjective, there exists $S\in B(H)$ such that $\varphi(S)=\frac{iI}{2}$. So
\begin{align*}
D(0, \varepsilon)&= \sigma_{\varepsilon}([iI\bullet \varphi^{-1}(\frac{iI}{2}),S]_{\ast})=\sigma_{\varepsilon}([\varphi(iI)\bullet \frac{iI}{2},\frac{iI}{2}]_{\ast})\nonumber\\
&=\sigma_{\varepsilon}(\frac{-1}{2}(\varphi(iI)+\varphi(iI)^{\ast})).
\end{align*}
It follows from Lemma 1.1 that, $\varphi(iI)^{\ast}=-\varphi(iI)$.\\

Now let $T\in B(H)$ be an arbitrary operator. Then
\begin{align*}
D(0, \varepsilon)&= \sigma_{\varepsilon}([iI\bullet T,S]_{\ast})=\sigma_{\varepsilon}([\varphi(iI)\bullet \varphi(T),\varphi(S)]_{\ast})\nonumber\\
&=\sigma_{\varepsilon}([\varphi(iI)\varphi(T)+\varphi(T)\varphi(iI)^{\ast},\frac{iI}{2}]_{\ast})\nonumber\\
&=\sigma_{\varepsilon}(\frac{iI}{2}(\varphi(iI)(\varphi(T)-\varphi(T)^{\ast})-(\varphi(T)-\varphi(T)^{\ast})\varphi(iI))).
\end{align*}
By Lemma 1.1, we have $\varphi(iI)(\varphi(T)-\varphi(T)^{\ast})-(\varphi(T)-\varphi(T)^{\ast})\varphi(iI)=0$. It follows from the surjectivity of $\varphi$ that, $\varphi(iI)B=B\varphi(iI)$ for all $B \in B_{a}(H)$ and so $\varphi(iI)B=B\varphi(iI)$ for all $B \in B_{s}(H)$. Since for every $A\in B(H)$; we have $A=A_{1}+A_{2}$, where $A_{1}$ and $A_{2}$ are self-adjoint elements, Hence $\varphi(iI)A=A\varphi(iI)$ holds true for all $A \in B(H)$, then $\varphi(iI)\in Z(B(H))$. Similarly, we have $\varphi^{-1}(iI)\in Z(B(H))$.\\

{\bf Claim 2.}  $\varphi$ preserves the self-adjoint elements in both
direction, and $\varphi(iT)^{\ast}=-\varphi(iT)$ for every $T\in B_{s}(H)$.\\
Let $T=T^{\ast}$ and $\varphi(S)=\frac{I}{2}$ for some $S\in B(H)$. We have
\begin{align*}
D(0, \varepsilon)&= \sigma_{\varepsilon}([S\bullet T,\varphi^{-1}(iI)]_{\ast})=\sigma_{\varepsilon}([\frac{I}{2}\bullet \varphi(T),iI]_{\ast})\nonumber\\
&=\sigma_{\varepsilon}(i(\varphi(T)-\varphi(T)^{\ast})).
\end{align*}
It follows from Lemma 1.1 that, $\varphi(T)-\varphi(T)^{\ast}=0$, and so $\varphi(T)=\varphi(T)^{\ast}$. Similarly, if $\varphi(T)=\varphi(T)^{\ast}$, then
\begin{align*}
D(0, \varepsilon)&= \sigma_{\varepsilon}([\varphi(\frac{I}{2})\bullet \varphi(T),\varphi(iI)]_{\ast})=\sigma_{\varepsilon}([\frac{I}{2}\bullet T,iI]_{\ast})\nonumber\\
&=\sigma_{\varepsilon}(i(T-T^{\ast})),
\end{align*}
so $T=T^{\ast}$. For
the second part of this claim, let $T\in B_{s}(H)$ and $\varphi(S)=I$ for some $S\in B(H)$, we have
\begin{align*}
D(0, \varepsilon)&= \sigma_{\varepsilon}([iT\bullet \varphi^{-1}(iI),S]_{\ast})=\sigma_{\varepsilon}([\varphi(iT)\bullet iI,\varphi(S)]_{\ast})\nonumber\\
&=\sigma_{\varepsilon}(2i(\varphi(iT)+\varphi(iT)^{\ast})).
\end{align*}
Again by Lemma 1.1, we see that $\varphi(iT)^{\ast}=-\varphi(iT)$ for every $T\in B_{s}(H)$.\\

{\bf Claim 3.}  $\varphi^{2}(I)\varphi(iI)=iI$ and $\varphi^{2}(iI)\varphi(I)=-I$. Hence $\varphi(I)$ and $\varphi(iI)$ are invertible.\\

We have
\begin{align*}
D(4i,\varepsilon)&=\sigma_{\varepsilon}(4iI)=\sigma_{\varepsilon}([I\bullet iI,I]_{\ast})=\sigma_{\varepsilon}([\varphi(I)\bullet \varphi(iI),\varphi(I)]_{\ast})\nonumber\\
&=\sigma_{\varepsilon}([\varphi(I)\varphi(iI)+\varphi(iI)\varphi(I)^{\ast},\varphi(I)]_{\ast})=
\sigma_{\varepsilon}(4\varphi^{2}(I)\varphi(iI)).
\end{align*}
It follows from Lemma 1.1 that, $\varphi^{2}(I)\varphi(iI)=iI$.\\
Similarly, note that
\begin{align*}
D(-4, \varepsilon)&=\sigma_{\varepsilon}(-4I)=\sigma_{\varepsilon}([I\bullet iI,iI]_{\ast})=\sigma_{\varepsilon}([\varphi(I)\bullet \varphi(iI),\varphi(iI)]_{\ast})\nonumber\\
&=\sigma_{\varepsilon}([\varphi(I)\varphi(iI)+\varphi(iI)\varphi(I)^{\ast},\varphi(iI)]_{\ast})=
\sigma_{\varepsilon}(4\varphi^{2}(I)\varphi(iI)).
\end{align*}
It follows that, again by lemma 1.1 $\varphi^{2}(iI)\varphi(I)=-I$.\\

Now, we define a map $\psi:B(H) \rightarrow B(H)$ by $\psi(T)=-i\varphi(I)\varphi(iI)\varphi(T)$ for all $T \in B(H)$. It is easy to see that $\psi$ is a
bijective map with $\psi(I)= I$ and $\psi(iI)= iI$, and also satisfies $\sigma_{\varepsilon}([T_{1}\bullet T_{2},T_{3}]_{\ast})=\sigma_{\varepsilon}([\psi(T_{1})\bullet \psi(T_{2}),\psi(T_{3})]_{\ast})$
for all $T_{1}, T_{2},T_{3}\in B(H)$. Furthermore, it is clear that $\psi$ preserves the self-adjoint elements in both direction.\\

{\bf Claim 4.} We have the following statments:\\
(i) $\sigma_{\frac{\varepsilon}{2}}([T,S]_{\ast})=\sigma_{\frac{\varepsilon}{2}}([\psi(T),\psi(S)]_{\ast})$ for every $T,S\in B(H)$.\\
(ii) $\psi(\frac{iI}{2})=\frac{iI}{2}$.\\
(iii) $\sigma_{\frac{\varepsilon}{2}}(T)=\sigma_{\frac{\varepsilon}{2}}(\psi(T))$ for every $T\in B(H)$.\\
(v) $\psi(iT)=i\psi(T)$ for every $T\in B_{s}(H)$.\\

(i) For every $T, S\in B(H)$, we have
\begin{align*}
\sigma_{\varepsilon}(2(TS-ST^{\ast}))&=\sigma_{\varepsilon}([I\bullet T,S]_{\ast})=\sigma_{\varepsilon}([\psi(I)\bullet \psi(T),\psi(S)]_{\ast})\nonumber\\
&=\sigma_{\varepsilon}(2(\psi(T)\psi(S)-\psi(S)\psi(T)^{\ast})).
\end{align*}
It follows that $\sigma_{\frac{\varepsilon}{2}}([T,S]_{\ast})=\sigma_{\frac{\varepsilon}{2}}([\psi(T),\psi(S)]_{\ast})$ for every $T,S\in B(H)$.\\

(ii) We have
\begin{align*}
D(-2,\varepsilon)&=\sigma_{\varepsilon}(-2I)=\sigma_{\varepsilon}([I\bullet iI,\frac{iI}{2}]_{\ast})\nonumber\\
&=\sigma_{\varepsilon}([\psi(I)\bullet \psi(iI),\psi(\frac{iI}{2}))]_{\ast})=\sigma_{\varepsilon}(4i\psi(\frac{iI}{2})).
\end{align*}
It follows that, by lemma 1.1 $\psi(\frac{iI}{2})=\frac{iI}{2}$.\\

(iii) For every $T\in B(H)$, by $(ii)$ we have
\begin{align*}
\sigma_{\frac{\varepsilon}{2}}(iT)&=\sigma_{\frac{\varepsilon}{2}}(\frac{iI}{2}T+T\frac{iI}{2})=\sigma_{\frac{\varepsilon}{2}}(\frac{iI}{2}T-T(\frac{iI}{2})^{\ast})\nonumber\\
&=\sigma_{\frac{\varepsilon}{2}}(\psi(\frac{iI}{2})\psi(T)-\psi(T)\psi(\frac{iI}{2})^{\ast})\nonumber\\
&=\sigma_{\frac{\varepsilon}{2}}(\psi(\frac{iI}{2})\psi(T)+\psi(T)\psi(\frac{iI}{2}))\nonumber\\
&=\sigma_{\frac{\varepsilon}{2}}(\frac{iI}{2}\psi(T)+\psi(T)\frac{iI}{2})=
\sigma_{\frac{\varepsilon}{2}}(i\psi(T)).
\end{align*}
It follows that, $\sigma_{\frac{\varepsilon}{2}}(T)=\sigma_{\frac{\varepsilon}{2}}(\psi(T))$ for every $T\in B(H)$.\\

(v) Note that $S(iT)-(iT)S^{\ast}$ is normal, where $T\in B_{s}(H)$ and $S\in B(H)$, so from this and Lemma 1.1(2) we get
\begin{align*}
\sigma(\psi(S)\psi(iT)-\psi(iT)\psi(S)^{\ast})&=\sigma(S(iT)-(iT)S^{\ast})=i\sigma(ST-TS^{\ast})\nonumber\\
&=i\sigma(\psi(S)\psi(T)-\psi(T)\psi(S)^{\ast})\nonumber\\
&=\sigma(\psi(S)(i\psi(T))-(i\psi(T))\psi(S)^{\ast}).
\end{align*}
By surjectivity of $\psi$ and lemma 1.3, we have $\psi(iT)=i\psi(T)$ for every $T\in B_{s}(H)$.\\

{\bf Claim 5.} There exists a unitary operator $U$ on $H$ such that $\psi(T)=\lambda UTU^{\ast}$ or $\psi(T)=\lambda UT^{t}U^{\ast}$ for every $T\in B_{s}(H)$, where $\lambda \in \{-1,1\}$.\\

The equality $\sigma_{\frac{\varepsilon}{2}}(T)=\sigma_{\frac{\varepsilon}{2}}(\psi(T))$ for every $T\in B(H)$, together Lemma 1.1(2),
implies that $P\in P(H)$ if and only if $\psi(P)$ is a projection.
Let $P, Q\in P(H)$ such that $PQ=QP=0$. It follows from claim 5(v) that
\begin{align*}
D(0, \frac{\varepsilon}{2})&=\sigma_{\frac{\varepsilon}{2}}([iP,Q]_{\ast})=\sigma_{\frac{\varepsilon}{2}}([\psi(iP),\psi(Q)]_{\ast})\nonumber\\
&=\sigma_{\frac{\varepsilon}{2}}( i(\psi(P)\psi(Q)+\psi(Q)\psi(P))),
\end{align*}
and consequently, $\psi(P)\psi(Q)+\psi(Q)\psi(P)=0$. Since $\psi(P)$ and $\psi(Q)$ are projection, then $\psi(P)\psi(Q)=\psi(Q)\psi(P)=0$. Conversely, if $\psi(P)$ and $\psi(Q)$ are projections such that $\psi(P)\psi(Q)=\psi(Q)\psi(P)=0$, then a similar discussion implies that $PQ=QP=0$. Thus, $\psi:P(H) \rightarrow P(H)$
is a bijective map which preserves the orthogonality in both directions. So, by \cite{sem}, there exists a unitary
or conjugate unitary operator $U$ on $H$ such that $\psi(P)=UPU^{\ast}$ for every $P\in P(H)$.\\
Now let $T\in B_{s}(H)$ and $x\in H$ be an unit arbitrary nonzero vector. First assume that $U$ is unitary. It follows from Lemma 1.1(2) and claim 4(v) that
\begin{align*}
D(0,\frac{\varepsilon}{2})+\sigma(iT( x\otimes x)+(x\otimes x)iT)&=\sigma_{\frac{\varepsilon}{2}}(iT(x\otimes x)+(x\otimes x)iT)\nonumber\\
&=\sigma_{\frac{\varepsilon}{2}}(iT(x\otimes x)-(x\otimes x)(iT)^{\ast})\nonumber\\
&=\sigma_{\frac{\varepsilon}{2}}(\psi(iT) \psi(x\otimes x)-\psi(x\otimes x)\psi(iT)^{\ast})\nonumber\\
&=\sigma_{\frac{\varepsilon}{2}}(i\psi(T) U(x\otimes x)U^{\ast}+U(x\otimes x)U^{\ast}i\psi(T))\nonumber\\
&=D(0,\frac{\varepsilon}{2})+\sigma(i\psi(T) U(x\otimes x)U^{\ast}+U(x\otimes x)U^{\ast}i\psi(T)).
\end{align*}
So $\sigma(T(x\otimes x)+(x\otimes x)T)=\sigma(\psi(T) U(x\otimes x)U^{\ast}+U(x\otimes x)U^{\ast}\psi(T))$. Therefore
\begin{align*}
2\left \langle Tx,x \right \rangle &=Tr(T(x\otimes x)+(x\otimes x)T)\nonumber\\
&=Tr(\psi(T) U(x\otimes x)U^{\ast}+U(x\otimes x)U^{\ast}\psi(T))\nonumber\\
&=2\left \langle U^{\ast}\psi(T)U,x \right \rangle.
\end{align*}

It follows that $\psi(T)= UTU^{\ast}$ for every $T\in B_{s}(H)$.\\

Now assume that $U$ is conjugate unitary. We define the map $J : H \rightarrow H$ by $J(\sum_{i \in \Lambda}\lambda_{i}e_{i})=\sum_{i\in \Lambda}\lambda_{i}\overline{e_{i}}$, where $\{e_{i}\}_{i\in \Lambda}$ is an orthonormal basis of $H$. It is easy to see that $J$ is conjugate unitary and $JT^{\ast}J =T^{t}$. Let $U=VJ$, then $V$ is unitary,
and $\varphi(T)=VJTJV^{\ast}=VT^{t}V^{\ast}$ for every $T\in B(H)$. \\

It is easy to see that maps $T \rightarrow T^{t}$ and $T \rightarrow U^{\ast}TU$ preserve the pseudo spectrum of skew Lie product, so
we may as well assume that $\psi(T)=T$ for every $T\in B_{s}(H)$.\\

{\bf Claim 6.} $\psi(iT)=iT$ for every $T\in B_{s}(H)$.\\

Let $x\in H$ be an arbitrary nonzero vector and $S=iT$, where $T\in B_{s}(H)$. It follows from Lemma 1.1(2) that
\begin{align*}
D(0,\frac{\varepsilon}{2})+\sigma(S( x\otimes x)+(x\otimes x) S)&=\sigma_{\frac{\varepsilon}{2}}(S (x\otimes x)+(x\otimes x) S)\nonumber\\
&=\sigma_{\frac{\varepsilon}{2}}(S(x\otimes x)-(x\otimes x)S^{\ast})\nonumber\\
&=\sigma_{\frac{\varepsilon}{2}}(\psi(S) \psi(x\otimes x)-\psi(x\otimes x)\psi(S)^{\ast})\nonumber\\
&=\sigma_{\frac{\varepsilon}{2}}(\psi(S) (x\otimes x)+(x\otimes x)\psi(S))\nonumber\\
&=D(0,\frac{\varepsilon}{2})+\sigma(\psi(S) (x\otimes x)+(x\otimes x)\psi(S)).
\end{align*}
 Hence $\sigma(S(x\otimes x)+(x\otimes x) S)=\sigma(\psi(S) (x\otimes x)+(x\otimes x)\psi(S))$. By Lemma 1.2,

$$\{0, \left \langle Sx,x \right \rangle \pm \sqrt{ \left \langle S^{2}x,x\right \rangle} \}=\{0, \left \langle \psi(S)x,x \right \rangle \pm \sqrt{ \left \langle \psi(S)^{2}x,x\right \rangle} \}.$$
Therefore, either
$$\left \langle Sx,x \right \rangle + \sqrt{ \left \langle S^{2}x,x\right \rangle}=\left \langle \psi(S)x,x \right \rangle + \sqrt{ \left \langle \psi(S)^{2}x,x\right \rangle}$$
and
$$\left \langle Sx,x \right \rangle - \sqrt{ \left \langle S^{2}x,x\right \rangle}=\left \langle \psi(S)x,x \right \rangle - \sqrt{ \left \langle \psi(S)^{2}x,x\right \rangle},$$
or
$$\left \langle Sx,x \right \rangle + \sqrt{ \left \langle S^{2}x,x\right \rangle}=\left \langle \psi(S)x,x \right \rangle - \sqrt{\left \langle \psi(S)^{2}x,x\right \rangle}$$
and
$$\left \langle Sx,x \right \rangle - \sqrt{ \left \langle S^{2}x,x\right \rangle}=\left \langle \psi(S)x,x \right \rangle + \sqrt{ \left \langle \psi(S)^{2}x,x\right \rangle}.$$
Combining the two equations in either case, we clearly get that $\left \langle Sx,x \right \rangle =\left \langle \psi(S)x,x \right \rangle$. Since $x\in H$ is an arbitrary unit vector, we conclude that $\psi(iT)=iT$ for every $T\in B_{s}(H)$.\\

{\bf Claim 7.} The result in the theorem holds.\\

Let $T\in B(H)$ be arbitrary. For any unit vector $x\in H$ and $\alpha > 0$, we have
\begin{align*}
i\alpha\sigma_{\frac{\delta}{\alpha}}((x\otimes x)T+T(x\otimes x))&=\sigma_{\delta}((i\alpha x\otimes x)T- T(i\alpha x\otimes x)^{\ast})\nonumber\\
&=\sigma_{\delta}(\psi(i\alpha x\otimes x)\psi(T)-\psi(T)\psi(i\alpha x\otimes x)^{\ast})\nonumber\\
&=\sigma_{\delta}((i\alpha x\otimes x)\psi(T)+\psi(T)(i\alpha x\otimes x))\nonumber\\
&=i\alpha\sigma_{\frac{\delta}{\alpha}}((x\otimes x)\psi(T)+\psi(T)(x\otimes x)),
\end{align*}
where $\delta=\frac{\varepsilon}{2}$. On the other hand
\begin{align*}
\sigma((x\otimes x)T+T(x\otimes x))&=\bigcap_{\alpha>0} \sigma_{\frac{\delta}{\alpha}}((x\otimes x)T+T(x\otimes x))\nonumber\\
&=\bigcap_{\alpha>0}\sigma_{\frac{\delta}{\alpha}}((x\otimes x)\psi(T)+\psi(T)(x\otimes x))\nonumber\\
&=\sigma((x\otimes x)\psi(T)+\psi(T)(x\otimes x)).
\end{align*}
Thus $\sigma((x\otimes x)T+T(x\otimes x))=\sigma((x\otimes x)\psi(T)+\psi(T)(x\otimes x))$.
Therefore, following the same argument as the one in the proof of
Claim 6,  one concludes that $ \left \langle Tx,x \right \rangle= \left \langle \psi(T)x,x \right \rangle$ for any nonzero vector $x\in H$. Hence $\psi(T)=T$, and therefore $\varphi(T)=SUTU^{\ast}$ or $\varphi(T)=SUT^{t}U^{\ast}$ for every $T \in B(H)$, where $S=(i \varphi(I)\varphi(iI))^{-1}$ or $S=(-i \varphi(I)\varphi(iI))^{-1}$. The proof is complete.\\
\end{proof}

We end this paper with the following theorem which characterizes bijective maps that satisfy
$$\sigma_{\varepsilon}(T_{1}\diamond T_{2}\circ_{\ast} T_{3})=\sigma_{\varepsilon}(\varphi(T_{1})\diamond \varphi(T_{2})\circ_{\ast} \varphi(T_{3})),~~ (T_{1}, T_{2}, T_{3} \in B(H)),$$
where $T\diamond S=TS^{\ast}+S^{\ast}T$ and $T\circ_{\ast} S=TS^{\ast}-ST$ for every $T, S\in B(H)$.

\begin{theorem}\label{l2}
Suppose that a bijective map $\varphi: B(H) \rightarrow B(H)$ satisfies
$$\sigma_{\varepsilon}(T_{1}\diamond T_{2}\circ_{\ast} T_{3})=\sigma_{\varepsilon}(\varphi(T_{1})\diamond \varphi(T_{2})\circ_{\ast} \varphi(T_{3})),~~ (T_{1}, T_{2}, T_{3} \in B(H)).$$
If $\varphi(iI)$ be anti-selfadjoint, then $\varphi^{2}(I)$ is invertible and there exist a unitary operator $U\in B(H)$ such that
 $\varphi(T)=\lambda(\varphi^{2}(I))^{-1}UTU^{\ast}$ or $\varphi(T)=\lambda(\varphi^{2}(I))^{-1}UT^{t}U^{\ast}$ for every $T \in B(H)$, where $\lambda \in \{-1,1\}$.
\end{theorem}

\begin{proof}
The proof breaks down into six claims.\\

{\bf Claim 1.}  $\varphi(I)^{\ast}=\varphi(I)\in Z(B(H))$.\\

Since $\varphi$ is surjective, there exist $S\in B(H)$ such that $\varphi(S)=I$. For every $T\in B(H)$, we have
\begin{align*}
D(0, \varepsilon)&= \sigma_{\varepsilon}(T\diamond S\circ_{\ast} I)=\sigma_{\varepsilon}(\varphi(T)\diamond \varphi(S)\circ_{\ast} \varphi(I))\nonumber\\
&=\sigma_{\varepsilon}(2\varphi(T)\varphi(I)^{\ast}-2\varphi(I)\varphi(T)).
\end{align*}
Let $T=S$, by Lemma 1.1 we can conclude that $\varphi(I)^{\ast}=\varphi(I)$. Since $\varphi$ is surjective, we
have $\varphi(I)\in Z(B(H))$.\\

{\bf Claim 2.}  $\varphi$ preserves the self-adjoint elements in both
direction.\\
Let $T=T^{\ast}$. We have
\begin{align*}
D(0, \varepsilon)&=\sigma_{\varepsilon}(I\diamond I\circ_{\ast} T)=\sigma_{\varepsilon}(\varphi(I)\diamond \varphi(I)\circ_{\ast} \varphi(T))\nonumber\\
&=\sigma_{\varepsilon}(2\varphi(I)^{2}(\varphi(T)^{\ast}-\varphi(T))).
\end{align*}
This implies that $\varphi(T)=\varphi(T)^{\ast}$. Similarly, if $\varphi(T)=\varphi(T)^{\ast}$, then $T=T^{\ast}$.\\

{\bf Claim 3.}  $\varphi^{2}(I)\varphi(iI)=iI$, that is $\varphi^{2}(I)$ is invertible.\\

We have
\begin{align*}
D(-4i,\varepsilon)&=\sigma_{\varepsilon}(-4iI)=\sigma_{\varepsilon}(I\diamond I\circ_{\ast} iI)=\sigma_{\varepsilon}(\varphi(I)\diamond \varphi(I)\circ_{\ast} \varphi(iI))\nonumber\\
&=\sigma_{\varepsilon}(-4\varphi^{2}(I)\varphi(iI)).
\end{align*}
It follows that, by lemma 1.1 $\varphi^{2}(I)\varphi(iI)=iI$.\\

Now, defining a map $\psi:B(H) \rightarrow B(H)$ by $\psi(T)=\varphi^{2}(I)\varphi(T)$ for all $T \in B(H)$. It is easy to see that $\psi$ is a
bijection with $\psi(iI)= iI$, and satisfies $\sigma_{\varepsilon}(T_{1}\diamond T_{2}\circ_{\ast} T_{3})=\sigma_{\varepsilon}(\psi(T_{1})\diamond \psi(T_{2})\circ_{\ast} \psi(T_{3}))$ for all $T_{1}, T_{2}, T_{3}\in B(H)$. Furthermore, for every $T,S\in B(H)$, we have
\begin{align*}
\sigma_{\varepsilon}(-2i(TS^{\ast}+S^{\ast}T))&=\sigma_{\varepsilon}(T\diamond S\circ_{\ast} iI)=\sigma_{\varepsilon}(\psi(T)\diamond \psi(S)\circ_{\ast} \psi(iI))\nonumber\\
&=\sigma_{\varepsilon}(-2i(\psi(T)\psi(S)^{\ast}+\psi(S)^{\ast}\psi(T))).
\end{align*}
It follows that, $\sigma_{\frac{\varepsilon}{2}}(TS^{\ast}+S^{\ast}T)=\sigma_{\frac{\varepsilon}{2}}(\psi(T)\psi(S)^{\ast}+\psi(S)^{\ast}\psi(T))$ for every $T,S\in B(H)$.\\

{\bf Claim 5.} There exists a unitary operator $U$ on $H$ such that $\psi(T)=\lambda UTU^{\ast}$ or $\psi(T)=\lambda UT^{t}U^{\ast}$ for every $T\in B_{s}(H)$, where $\lambda \in \{-1,1\}$.\\

It is clear that $\psi$ preserves the self-adjoint elements in both direction, so $\psi|_{B_{S}(H)}:B_{S}(H) \rightarrow B_{S}(H)$ is a bijective map which satisfies
$\sigma_{\frac{\varepsilon}{2}}(TS+ST)=\sigma_{\frac{\varepsilon}{2}}(\psi(T)\psi(S)+\psi(S)\psi(T))$ for every $T,S\in B_{s}(H)$. So, by Theorem 1.4, there exists a unitary operator $U$ on $H$ such that $\psi(T)=\lambda UTU^{\ast}$ or $\psi(T)=\lambda UT^{t}U^{\ast}$ for every $T\in B_{s}(H)$, where $\lambda \in \{-1,1\}$.\\

Since the maps $T \rightarrow T^{t}$ and $T \rightarrow U^{\ast}TU$ preserve the pseudo spectrum of $TS^{\ast}+S^{\ast}T$,
we may as well assume that $\psi(T)=T$ for every $T\in B_{s}(H)$.\\

{\bf Claim 6.} $\psi(T)=T$ for every $T\in B(H)$.\\

Let $T\in B(H)$ be arbitrary. For any vector $x\in H$ and $\alpha > 0$, we have
\begin{align*}
\alpha\sigma_{\frac{\delta}{\alpha}}(T(x\otimes x)+(x\otimes x)T)&=\sigma_{\delta}(T(\alpha x\otimes x)+(\alpha x\otimes x)T)\nonumber\\
&=\sigma_{\delta}(\psi(T) \psi(\alpha x\otimes x)+\psi(\alpha x\otimes x)\psi(T))\nonumber\\
&=\sigma_{\delta}(\psi(T)(\alpha x\otimes x)+(\alpha x\otimes x)\psi(T))\nonumber\\
&=\alpha\sigma_{\frac{\delta}{\alpha}}(\psi(T)(x\otimes x)+(x\otimes x)\psi(T)),
\end{align*}
where $\delta=\frac{\varepsilon}{2}$. On the other hand
\begin{align*}
\sigma(T(x\otimes x)+(x\otimes x) T)&=\bigcap_{\alpha>0} \sigma_{\frac{\delta}{\alpha}}(T(x\otimes x)+(x\otimes x) T)\nonumber\\
&=\bigcap_{\alpha>0}\sigma_{\frac{\delta}{\alpha}}(\psi(T)(x\otimes x)+(x\otimes x)\psi(T))\nonumber\\
&=\sigma(\psi(T)(x\otimes x)+(x\otimes x)\psi(T)).
\end{align*}
Thus $\sigma(T(x\otimes x)+(x\otimes x) T)=\sigma(\psi(T)(x\otimes x)+(x\otimes x)\psi(T))$.
By the same argument of proof Claim 7 in Theorem 2.1, we conclude that $ \left \langle Tx,x \right \rangle= \left \langle \psi(T)x,x \right \rangle$ for any nonzero vector $x\in H$. As a result, $\psi(T)=T$, and therefore $\varphi(T)=\lambda(\varphi^{2}(I))^{-1}UTU^{\ast}$ or $\varphi(T)=\lambda(\varphi^{2}(I))^{-1}UT^{t}U^{\ast}$ for every $T \in B(H)$. The proof is complete.
\end{proof}


\begin{thebibliography}{99}
\bibitem{abdel} Z. Abidine Abdelali and H. Nkhaylia, \textit{Maps preserving the pseudo spectrum of skew
triple product of operators}, Linear and Multilinear Algebra, \textbf{67}(11) (2019), 2297--2306.
\bibitem{al} E. Alzedani and M. Mabrouk, \textit{Maps preserving the spectrum of skew lie oroduct of operators},
 Kragujevac Journal of Mathematics, \textbf{64}(4) (2022), 525--532.
 \bibitem{Aup} B. Aupetit, \textit{Spectrum-preserving linear mappings between Banach algebras or Jordan–Banach algebras}, J. London
Math. Soc. \textbf{62} (2000) 917--924.
\bibitem{Bai} Z. Bai, S. Du, \textit{Maps preserving products $XY+Y X^{\ast}$ on von Neumann algebras}, J. Math. Anal. Appl. \textbf{386} (2012), 103--109.
\bibitem{Ben1} M. Bendaoud, A. Benyouness and M. Sarih, \textit{Preservers of pseudo spectral
radius of operator products}, Linear Algebra Appl. \textbf{489} (2016), 186--198.
\bibitem{Ben2} M. Bendaoud, A. Benyouness and M. Sarih, \textit{Nonlinear maps preserving
the pseudo spectral radius of skew semi-triple products of operators}, Acta Sci.
Math. (Szeged), \textbf{84} (2018), 39--47.
\bibitem{Ben3} M. Bendaoud, A. Benyouness and M. Sarih, \textit{Preservers of pseudo spectra
of operator Jordan triple products}, Oper. Matrices.
 \textbf{1} (2016), 45--56.
 \bibitem{Bourhim2} A. Bourhim, J. Mashreghi, \textit{A survey on preservers of spectra and local
spectra}, Contemp Math. \textbf{45}, 45--98 (2015).
\bibitem{cu3} J. Cui, V. Forstall, C.K. Li, V. Yannello, \textit{Properties and preservers of the pseudospectrum}, Linear
Algebra Appl. \textbf{436} (2012), 316--325.
\bibitem{cu2} J. Cui, C.K. Li, Y.T. Poon, \textit{Pseudospectra of special operators and pseudospectrum preservers}, J. Math. Anal. Appl. \textbf{419} (2014), 1261--1273.
\bibitem{die} J. Dieudonné, \textit{Sur une généralisation du groupe orthogonal a quatre variables}, Arch. Math. \textbf{1} (1994), 282--287.
\bibitem{fro} G. Frobenius, \textit{Ueber die Darstellung der endlichen Gruppen durch lineare Substitutionen}, Berl Ber. Appl. \textbf{203} (1897), 994--1015.
\bibitem{gli} A.M. Gleason, \textit{A characterization of maximal ideals}. J. Analyse Math. \textbf{19} (1967) 171--172.
\bibitem{jaf} A.A. Jafarian, \textit{A survey of invertibility and Spectrum-preserving linear maps}, Bulletin of the Iranian Mathematical Society. \textbf{35}(2) (2009), 1--10.
\bibitem{jaf2} A.A. Jafarian, A.R. Sourour, \textit{Spectrum-preserving linear maps}, J. Funct. Anal. \textbf{66} (1986) 255--261.
\bibitem{kah} J.P. Kahane, W. Zelazko, \textit{A characterization of maximal ideals in commutative Banach algebras}. Studia Math. \textbf{29} (1968) 339--343.
\bibitem{kap} I. Kaplansky, Algebraic and analytic aspects of operator algebras, Conference Board of the
Mathematical Sciences Regional Conference Series in Mathematics, No. 1. Providence (RI):
American Mathematical Society; 1970.
\bibitem{kum} G.K. Kumar, S.H. Kulkarni, \textit{Linear maps preserving pseudospectrum and condition spectrum},  Banach J. Math. Anal. \textbf{6} (2012), 45--60.
\bibitem{li} C.K. Li, N.C Sze, \textit{Product of operators and numerical range preserving maps},  Studia Math. \textbf{174} (2006), 169--182.
\bibitem{mar} M. Marcus, B.N. Moyls, \textit{Linear transformations on algebras of matrices}, Canad.
J. Math. \textbf{11} (1959) 61-66.
\bibitem{parvin} R. Parvinianzadeh, J. Pazhman, \textit{A Collection of local spectra preserving maps}, Math.
anal. convex optim. \textbf{3}(1) (2022) 49--60.
\bibitem{sem} P. Šemrl, \textit{Maps on idempotent operators}, Studia Math. \textbf{169} (2005) 21--44.
\bibitem{Sourour} A.R. Sourour, \textit{Invertibility preserving linear maps on L(X)}, Trans.
Amer. Math. Soc. \textbf{348} (1996) 13--30.
\bibitem{tre} L.N. Trefethen and M. Embree, \textit{Spectra and pseudospectra}, Princeton Univ. Press, Princeton, NJ, 2005.
\bibitem{zel} W. Zelazko, \textit{A characterization of multiplicative linear functionals in complex Banach algebras},
Studia Math. \textbf{30} (1968) 83--85.
\end{thebibliography}
\end{document}